\documentclass[a4paper, 11pt]{article}

\usepackage{calc, amsmath, amsthm, a4, latexsym, amssymb, layout, subfigure, color, bm, xspace, tabularx}
\usepackage{hyperref}
\usepackage{graphicx}
\usepackage[active]{srcltx}

\definecolor{comcolor}{rgb}{0.9,0.3,0.3}
\definecolor{starcolor}{rgb}{0.3,0.3,0.9}
\definecolor{hscolor}{rgb}{0.9,0.6,0.5}

\newtheorem{thm}{Theorem}[section]
\newtheorem{lemma}[thm]{Lemma}
\newtheorem{corollary}[thm]{Corollary}
\newtheorem{prop}[thm]{Proposition}

\theoremstyle{definition}

\newtheorem{example}[thm]{Example}

\newtheorem{remark}[thm]{Remark}

\newcommand{\be}[1]{\begin{equation}\label{#1}}
\newcommand{\ee}{\end{equation}}
\newcommand{\ba}{\begin{array}}
\newcommand{\ea}{\end{array}}
\newcommand{\bal}{\begin{aligned}}
\newcommand{\eal}{\end{aligned}}

\newcommand{\ie}{i.e.\@\xspace}
\newcommand{\R}{\mathbb{R}}

\newcommand{\N}{\mathbb{N}}

\newcommand{\E}{\mathbb{E}}
\newcommand{\p}{\mathbb{P}}

\newcommand{\calF}{\mathcal{F}}
\newcommand{\calG}{\mathcal{G}}

\newcommand{\calZ}{\mathcal{Z}}

\newcommand{\Var}{\mathrm{Var}}

\newcommand{\1}{1\hspace{-0.098cm}\mathrm{l}}

\newcommand{\la}{\lambda}
\newcommand{\Om}{\Omega}

\newcommand{\invis}[1]{}

\newcommand{\ra}{\rightarrow}

\newcommand{\ssup}[1] {{\scriptscriptstyle{({#1}})}}

\newcommand{\cD}{\mathcal{D}}

\newcommand{\cF}{\mathcal{F}}
\newcommand{\cG}{\mathcal{G}}

\newcommand{\cZ}{\mathcal{Z}}

\newcommand{\IP}{\mathbb{P}}

\newcommand{\II}{\mathbb{I}}
\newcommand{\IF}{\mathbb{F}}

\newcommand{\IN}{\mathbb{N}}

\newcommand{\IZ}{\mathbb{Z}}

\newcommand{\IE}{\mathbb{E}}
\newcommand{\iN}{\in\IN}

\newcommand{\imp}{\mathrm{imp}}

\newcommand{\ind}{\1}

\newcommand{\dd}{\mathrm{d}}                                
\newcommand{\nor}{\theta} 

\begin{document}

\vglue20pt 

\begin{center}
 {\huge \bf Robust analysis of preferential\\[.2cm] attachment models with fitness}
\end{center}

\bigskip
\bigskip


\centerline{{\Large{\sc Steffen Dereich$^{1}$} and {\sc Marcel Ortgiese$^2$}}}
\bigskip

\begin{center}\it
\parbox[c][3cm][t]{0.5\textwidth}{
\begin{center}
$^1$Institut f\"ur Mathematische Statistik\\
Westf.\ Wilhelms-Universit\"at M\"unster\\
Einsteinstra\ss{}e 62\\
48149 M\"unster\\
Germany \\[2mm]
\end{center}
}%
\hfill\parbox[c][3cm][t]{0.5\textwidth}{
\begin{center}
$^2$Institut f\"ur Mathematik\\
Technische Universit\"at Berlin\\
Str.\ des 17.\ Juni 136\\
10623 Berlin\\
Germany\\
\end{center}
}

\vspace{1cm}

{\rm 14 February 2013}\\
\end{center}

\bigskip
\bigskip

{\leftskip=1truecm
\rightskip=1truecm
\baselineskip=15pt
\small

\noindent{\slshape\bfseries Abstract.} The preferential attachment network with fitness is a dynamic random graph model. New vertices are introduced consecutively and a
new vertex is attached to an old vertex with probability proportional to the degree of the old one multiplied by a random fitness. We concentrate on the typical behaviour of the graph by calculating the fitness distribution of a vertex chosen proportional to its degree. For a particular variant of the model, this analysis was first carried out by Borgs, Chayes, Daskalakis and Roch. However, we present a new method, which is robust in the sense that it does not depend on the exact specification of the attachment law. In particular, we show that a peculiar phenomenon, referred to as Bose-Einstein condensation, can be observed in a wide variety of models.
Finally, we also compute the joint degree and fitness distribution of a uniformly chosen
vertex.
\bigskip

\noindent{\slshape\bfseries Keywords.} Barab\'asi-Albert model, 
power law, scale-free network, nonlinear preferential attachment, nonlinearity, dynamic random graph, condensation.
\bigskip

\noindent
{\slshape\bfseries 2010 Mathematics Subject Classification.} Primary 05C80 Secondary 60G42, 90B15 

}

\newpage

\section{Introduction}

Preferential attachment  models were popularized by~\cite{BarabasiAlbert_1999} 
as a possible model for complex networks such as the world-wide-web. 
The authors observed that a simple mechanism can explain the occurrence of power law degree distributions in real world networks. Often networks are the result of a continuous dynamic process:  new members enter social networks or new web pages are created and linked to popular old ones. In this process new vertices prefer to establish links to old vertices that are well connected. 
Mathematically, one considers a sequence of random graphs (\emph{random dynamic network}), where new vertices are introduced consecutively and then connected to each old vertex 
with a probability proportional to the degree of the old vertex. This rather simple mechanism leads to networks with power law degree distributions and thus offers
an explanation for their occurrence, see e.g.~\cite{BRST} for a mathematical account.

There are many variations of the classic model to address different
shortcomings, see e.g.~\cite{RemcoNotes} for an overview.
For example, a more careful analysis of the classical model
 shows that one can observe a 
``first to market''-advantage, where 
from a certain point onwards the vertex with maximal degree will
always remain maximal, see e.g.~\cite{DereichMoertes09}. 
Clearly, this is not the only possible scenario observed in real networks.
One possible improvement is to model the fact that 
vertices have an intrinsic quality or fitness, 
which would allow even younger vertices to overtake
old vertices in popularity.

Introducing fitness has a significant effect on the network formation. In particular, it may provoke condensation effects as indicated in~\cite{BB_2001}. 
 A first mathematically rigorous analysis  was carried out in~\cite{BCDR07} for the following variant of the model:
 First every (potential) vertex $i\in\IN$ is assigned  an independent identically distributed (say $\mu$-distributed) fitness~$\cF_i$. Starting with the network $\cG_1$ consisting of the single vertex $1$
with a self-loop, the network is formed as follows.
Suppose we have constructed the graph $\calG_n$ with vertices~$\{1,\dots,n\}$, 
then we obtain $\cG_{n+1}$ by 
\begin{itemize}
\item insertion of the vertex $n+1$ and 
\item insertion of a single edge linking up the new vertex to the old vertex $i\in\{1,\dots,n\}$ with probability proportional to
\begin{align}\label{eq:1710-1}
\cF_i\,\deg_{\cG_n}(i),
\end{align}
\end{itemize}
where $\deg_\cG(i)$ denotes the degree of vertex $i$ in a graph $\calG$.

In~\cite{BCDR07}, the authors
compute the asymptotic
fitness distribution of a vertex chosen proportional to its degree. This limit distribution is either absolutely continuous with respect to $\mu$ (``\emph{fit-get-richer phase}'') or has a singular component that puts mass on the essential supremum of $\mu$ (``\emph{condensation phase}'' or ``\emph{Bose-Einstein phase}'').  In the condensation phase a positive fraction of mass is shifted towards the essential
supremum of~$\mu$. 

The analysis in~\cite{BCDR07} uses a coupling argument with a 
generalized urn model, which was investigated by~\cite{Janson_2004}
using in turn a coupling with a multitype branching process.
A more direct approach was presented in~\cite{Bhamidi}, who
explicitly couples the random graph model with a multitype branching
process and then uses classical results, see e.g.~\cite{JagersNerman_1996}, to complete
the analysis. Both results rely very much on the particularities of the model specification. This is in strong contrast to the physicists' intuition which suggests that explicit details of the model specification do not have an impact. 

The aim of the article is to close or at least reduce this gap significantly.
We present a new approach to calculating fitness distributions, 
which is robust in the sense that it does not rely on the exact details of
the attachment rule. 
In particular, we show that the
condensation phenomenon can be observed in a wide range of variations
of the model.

What makes the preferential attachment model with fitness more difficult
to analyse than classic preferential attachment models is that the 
normalisation, obtained by summing the weights
in~(\ref{eq:1710-1}) over all vertices $i$,  is neither deterministic nor
is it linear in the degrees.

In the framework of the classical preferential attachment model, 
there are several approaches to specify
the model fairly robustly. 
A rather general approach to calculate degree distributions
 in the case of a constant normalisation is
presented in~\cite{HagbergWiuf_2006}, where only a (linear) recursion for
the degree sequence is assumed. However,
the approach is restricted to a deterministic out-degree.
For a linear model, the requirement of a deterministic normalisation
can be relaxed.
For example in~\cite{CooperFrieze_2003}, apart from 
more complicated update rules, the out-degree
of a new vertex is also allowed to be random (albeit of bounded degree). 
Similarly, in~\cite{Jordan_2006} it is only assumed that the out-degree distribution
has exponential moments. However, in these cases even though the normalisation
is random, it is rather well concentrated around its mean. A particular interesting
variant is when the out-degree is heavy-tailed as analysed in~\cite{Deijfenetal_2009}.
Here, the fluctuations of the normalisation around its mean start interfering
and alter the degree distributions significantly.

For non-linear preferential attachment models, a particular elegant way of dealing with a random normalisation
 is to establish a coupling
with a branching process, which implicitly takes care of the problem,
see for example the survey~\cite{Bhamidi}. 
This also includes models with sublinear preferential attachment rules,
see e.g.~\cite{RTV_07}.
A generalisation of the model with fitness is presented in~\cite{jordan_geometric_2012}, 
where the attractiveness of a vertex is a function of a random location 
in some metric space. However, in that setting the full analysis is only carried out when
the metric space is finite, which corresponds to only finitely many different values
for the random fitness in our model. 

Our approach shows a new way of dealing with the 
normalisation constant using a bootstrapping argument.
The idea is to start with a bound $\nor$ on 
the normalisation, from which we deduce a new bound $T(\nor)$. 
Then, by a continuity argument, we deduce that the correct limit
of the normalisation is a fixed point of $T$. We stress that the mapping $T$ is new and has not appeared in the physics literature on complex networks with fitness yet.

In particular, our proofs show that the condensation effect can be observed irrespectively
of the fine details of the model.
The phenomenon of Bose-Einstein condensation seems to have a
universal character, for an overview of further models  see~\cite{DM_emer_2012}.
The precise analysis of the dynamics in a closely related model are carried
out in~\cite{Dereich_13}.

\section{Definitions and main results}

We consider a dynamic graph model with fitness. Each vertex $i\in\IN$ is assigned  an independent $\mu$-distributed fitness~$\cF_i$, where $\mu$ is a compactly supported distribution  on the Borel sets of $(0,\infty)$ that is not a Dirac-distribution. We call $\mu$ the \emph{fitness distribution}.

We measure the importance of a vertex $i$ in a directed graph $\cG$ by its \emph{impact}
$$\imp_{\cG}(i):= 1+\text{ indegree of $i$ in $\cG$}.$$ 
For technical reasons, we set $\imp_{\cG}(i)=0$, if $i$ is not a vertex of $\cG$.

The complex network is represented by a sequence $(\cG_n)_{n\iN}$ of random directed multigraphs without loops that is built according to the following rules. Each graph $\cG_n$ consists of $n$ vertices labeled by $1,\dots,n$. The first graph consists of the single vertex~$1$ and no edges. Further, 
given $\cG_n$, the network $\cG_{n+1}$ is formed by  carrying out the following two steps:
\begin{itemize}\item Insertion of the vertex $n+1$.
\item Insertion of directed edges $n+1\to i$ for each old vertex $i\in\{1,\dots,n\}$ with intensity proportional to
\begin{align}\label{connect}
\cF_i \cdot \imp _{\cG_n} (i).
\end{align}
  \end{itemize}

Note that this is not a unique description of the network formation. We still need to clarify the explicit rule how new vertices connect to  old ones. We will do this in terms of the \emph{impact evolutions}: for each $i\in\IN$, we consider the process $\cZ(i)=(\cZ_n(i))_{n\in\IN}$ defined by
$$
\cZ_n(i):=\imp_{\cG_n}(i).
$$
Since all edges point from younger to older vertices and since in each step all new edges attach to the new vertex, the sequence  $(\cG_n)_{n\in\IN}$ can be recovered from the impact evolutions $(\cZ(i):i\in\IN)$. Indeed, for any $i,j,n\in\IN$ with $i<j\leq n$ there are exactly
$$
\Delta \cZ_{j-1}(i):= \cZ_j(i)-\cZ_{j-1}(i)
$$
links pointing from $j$ to $i$ in $\cG_n$.
Note that each impact evolution $\cZ(i)$ is monotonically increasing, $\N_0$-valued and satisfies $\cZ_n(i)=\ind_{\{n=i\}}$ for $n\leq i$, and any choice for the impact evolutions with these three properties describes uniquely a dynamic graph model.
We state the assumptions in terms of the impact evolutions. For a discussion of relevant examples, we refer the reader to the discussion below.

{\bf Assumptions.} 
Let $\la > 0$ be a parameter and define 
$$
\bar\cF_n=\frac 1 {\lambda n} \sum_{j=1}^n \cF_j\,\imp_{\cG_n}(j)=\frac1{\la n}\langle \cF, \cZ_n\rangle,
$$
where $\cZ_n:=(\cZ_n(i))_{i\in\IN}$.

We assume that the following three conditions are satisfied:

\begin{tabularx}{\textwidth}{ p{.8cm} >{\raggedright\arraybackslash}X}
 {\bf (A1)}
& \[ \E [ \Delta \calZ_{n}(i)  | \calG_n] = \frac {\calF_i \,\cZ_n(i)}{n \bar \calF_n} .\] \\
 {\bf (A2)} & There exists a constant $C^{\rm var}$ such that 
\[ \Var ( \Delta \calZ_{n}(i)  | \calG_n ) \leq C^{\rm var} \E [ \Delta \calZ_{n}(i)  | \calG_n] .  \]
  \\
 {\bf (A3)} & Conditionally on $\calG_n$, for $i \neq j$, we assume that
$\Delta \calZ_{n}(i)$ and $\Delta \calZ_{n}(j)$ are negatively correlated. 
\end{tabularx}

By assumption the essential supremum of $\mu$ is finite and strictly positive, say $s$. Since the model will still satisfy assumptions (A1) - (A3), if we replace $\cF_i$ by $\cF_i'=\cF_i/s$, we can and will assume without loss of generality that 

\begin{tabularx}{\textwidth}{ p{.8cm} >{\raggedright\arraybackslash}X}
{\bf (A0)} &   \vspace{-5mm}\[ \mathrm{ess\,sup}(\mu)=1 . \]
\end{tabularx}

\begin{remark}\label{re:outdeg} Assumptions (A1)-(A3) guarantee that 
 the total number of edges in the system is of order $\la n$, see Lemma~\ref{le:outdeg}.
\end{remark}

Let us give two examples that satisfy our assumptions.

\begin{example}\label{ex:Pois} \emph{Poisson outdegree (M1).}
The definition depends on a parameter $\lambda >0$. In model (M1), given $\cG_n$, the new vertex $n+1$ establishes for each old vertex $i\in\{1,\dots,n\}$  an independent  Poisson-distributed number of links $n+1\rightarrow i$ with parameter
$$
\frac {\cF_i\, \cZ_{n}(i)}{n\, \bar\cF_n}.
$$
Note that the conditional outdegree of a new vertex $n+1$, given~$\cG_n$, is Poisson-distributed  with parameter~$\lambda$.
\end{example}

\begin{example}\label{ex:multi} \emph{Fixed outdegree (M2).}
The definition relies on a parameter $\lambda\iN$ denoting the deterministic outdegree of new vertices. Given $\cG_n$, the number of edges connecting $n+1$ to the individual old vertices $1,\dots,n$ forms a multinomial random variable with parameters $\lambda$ and 
$$
\Bigl( \frac {\cF_i\, \calZ_{n}(i)}{\lambda n\, \bar\cF_n}\Bigr)_{i=1,\dots,n}, \text{ where }\bar\cF_n=\frac 1 {\lambda n} \sum_{i=1}^n \cF_i\,\calZ_n(i).
$$ 
The model (M2) with  $\lambda=1$ is the one  analysed in \cite{BCDR07}.\medskip
\end{example}

We analyse a sequence of random measures $(\Gamma_n)_{n\iN}$ on $[0,1]$ given by 
$$
\Gamma_n= \frac 1n \sum_{i=1}^n \cZ_n(i)\, \delta_{\cF_i}
$$
the \emph{impact distributions}.
These measures describe the relative impact of fitnesses. Note also that, up to normalisation, 
$\Gamma_n$ is the distribution of the fitness of a vertex chosen proportional 
to its impact.

\begin{thm}\label{thm:main} Suppose that Assumptions (A0)-(A3) are satisfied. If $\int \frac{f}{1-f}\,\mu(df)\geq \lambda$, we denote by $\nor^*\geq 1$ the unique value with
$$
\int \frac f{\nor^*-f}\,\mu(df)=\lambda
$$
and set otherwise $\nor^*=1$. One has
$$
\lim_{n\to\infty} \bar\cF_n=\nor^*, \text{ almost surely}
$$
and we distinguish two regimes:
\begin{itemize}
\item[(i)] {\bf Fit-get-richer phase.} Suppose that  $\int \frac{f}{1-f}\,\mu(df)\geq \lambda$.  $(\Gamma_n)$ converges, almost surely, in the weak$^*$ topology to $\Gamma$, where
$$
\Gamma(df)= \frac{\nor^*}{\nor^*-f} \,\mu(df) 
$$
 \item[(ii)] {\bf Bose-Einstein phase.} Suppose that  $\int \frac{f}{1-f}\,\mu(df)< \lambda$.  $(\Gamma_n)$ converges, almost surely, in the weak$^*$ topology to $\Gamma$, where 
$$ \Gamma(df) = \frac{1}{1- f}\, \mu(df) + \Big(1 + \la - \int_{[0,1)} \frac{1}{1-f} \mu(df)\Big) \, \delta_1.
$$
\end{itemize}
\end{thm}

\begin{remark}
In particular, the two phases can be characterized as follows.
In the  \emph{Fit-get-richer phase}, \ie  if $\int\frac {f}{1-f}\,\mu(dx)\geq \lambda$, then the limit of $(\Gamma_n)$  is absolutely continuous with respect to $\mu$.
However, in the \emph{Bose-Einstein-phase}, \ie if $\int\frac {f}{1-f}\,\mu(dx)<\lambda$, then the limit of  $(\Gamma_n)$ is not absolutely continuous with respect to $\mu$, but has an atom in~$1$. The explanation for this phenomenon is that a positive fraction
of newly incoming edges connects to vertices with fitness that is closer 
and closer to the essential supremum of the fitness distribution $\mu$, 
which in the limit amounts to an atom at the essential supremum.
\end{remark}

Next, we restrict attention to vertices with a fixed impact $k\in\N$. For $n\in\N$ we consider the random measure 
\[ \Gamma^\ssup{k}_n := \frac 1n \sum_{i=1}^n \ind_{\{\cZ_n(i)=k\}} \delta_{\cF_i} , \]
representing -- up to normalisation -- the random fitness of a uniformly chosen vertex with impact $k$. 

To prove convergence of $(\Gamma_n^\ssup {k})$, we need additional assumptions. Indeed, so far our assumptions admit models for which vertices are always connected by multiple edges in which case there would be no vertices with impact $2$.

We will work with the following assumptions:\smallskip

\begin{tabularx}{\textwidth}{ p{.8cm} >{\raggedright\arraybackslash}X}
{\bf (A4)} & $\forall k\in\N$: $\sup_{i=1,\dots,n} \ind_{\{Z_n(i)= k\}}\, n \,\IP(\Delta \cZ_n(i)\geq 2|\cG_n)\to0$,  \   a.s.  \\[.3cm]
{\bf (A4')}  & $\forall k\in\N$: $\sup_{i=1,\dots,n} \ind_{\{Z_n(i)= k\}}\, n \bigl|\,\IP(\Delta \cZ_n(i)=1|\cG_n)- \frac {\cF_i Z_n(i)}{n\,\bar\cF_n}\bigr|\to 0$, \ a.s.\\[.35cm]
 \end{tabularx}
Further we impose an additional assumption on the correlation structure:\smallskip

\begin{tabularx}{\textwidth}{ p{.8cm} >{\raggedright\arraybackslash}X}
{\bf (A5)} & Given $\calG_n$, the collection $\{ \Delta \calZ_n(i) \}_{i = 1}^n$
is negatively quadrant dependent in the sense that for any $i \neq j$,
and any $k,l \in \N$
\[ \p \{ \Delta \calZ_n(i) \leq k ; \Delta \calZ_n(j) \leq \ell | \calG_n\}
 \leq \p \{ \Delta \calZ_n(i) \leq k | \calG_n \} \p \{ \Delta \calZ_n(j) \leq \ell | \calG_n\}
\]
\end{tabularx}

\begin{remark}
Note that both Examples~\ref{ex:Pois} and~\ref{ex:multi} also satisfy these additional 
assumptions. 
Moreover, under Assumption (A1), Assumptions (A4) and  (A4') are equivalent to\smallskip

\begin{tabularx}{\textwidth}{ p{.85cm} >{\raggedright\arraybackslash}X}
{\bf (A4'')} &  $\forall k\in\N$:  $\sup_{i=1,\dots,n} \ind_{\{Z_n(i)= k\}}\, n \,
\E[ \Delta Z_n(i) \1_{\{ \Delta Z_n(i) \geq 2 \}} | \calG_n] \to0$, \ a.s.\\[.4cm]
\end{tabularx}
\end{remark}

\begin{thm}\label{thm:2} Suppose that Assumptions (A0), (A4), (A4') and (A5) are satisfied and that for some $\theta^*\in[1,\infty)$
$$
\lim_{n\to\infty} \bar\cF_n=\theta^*, \ \text{ almost surely.}
$$
Then one has that, almost surely, 
$(\Gamma_n^\ssup{k})$ converges in the weak$^*$ topology to $\Gamma^\ssup{k}$, where
\begin{equation}\label{eq:2811-1}   \Gamma^\ssup{k} (df) = \frac{1}{k+ \frac{\theta^*}{f}}  \frac{\theta^*}{f} \prod_{i=1}^{k-1} \frac{i}{i+ \frac{\theta^*}{f}}
 \mu(df) \end{equation}
\end{thm}


The theorem immediately allows to control the number of vertices with impact $k\in\N$. Let 
$$
p_n(k):=\frac 1n \sum_{i=1}^n \ind_{\{\cZ_n(i)=k\}} = \Gamma_n^\ssup{k}([0,1]).
$$

\begin{corollary}Under the assumptions of Theorem~\ref{thm:2}, one has that 
\[ \lim_{n\to\infty} p_n(k) = \int_{(0,1]} \frac{1}{k+ \frac{\theta^*}{f}}  \frac{\theta^*}{f} \prod_{i=1}^{k-1} \frac{i}{i+ \frac{\theta^*}{f}}
 \mu(df), \ \text{ almost surely.}\]	
\end{corollary}

\subsubsection*{Outline of the article}

 Section~\ref{se:prelim} starts with preliminary considerations. In particular, it introduces a stochastic approximation argument which among other applications also appeared in the context  of generalized
urn models, see e.g.\ the survey~\cite{PemantleSurvey}. In preferential
attachment models, these techniques only seem to have been used directly
in~\cite{jordan_geometric_2012}.
Roughly speaking, key quantities are expressed as approximations to stochastically perturbed differential equations. The perturbation is asymptotically negligible and one obtains descriptions by differential equations that are typically referred to as \emph{master equations}.

Section~\ref{se:bootstrap} is concerned with the proof of Theorem~\ref{thm:main}. Here the main task is to prove convergence of the random normalisation $(\bar\cF_n)$. This goal is achieved via a bootstrapping argument. Starting with an upper bound on $(\bar \cF_n)$ of the form 
$$
\limsup_{n\to\infty} \bar\cF_n\leq \theta, \ \text{ almost surely},
$$
we show that this statement remains true when replacing $\theta$ by
\begin{equation}\label{eq:T} T(\nor)= 1+ \frac 1\lambda \int  \frac {\nor-1}{\nor-f}\,f \,\mu(df).\end{equation}
Iterating the argument yields convergence to a fixed point. 
The mapping $T$ always has the fixed point $\nor = 1$. 
Moreover, it has
a second fixed point $\nor^* > 1$ if and only if $\int \frac{x}{\nor-x}\mu(dx) >\la$, 
which corresponds to the fit-get-richer phase.
In this case, one can check that only the larger fixed point $\nor^*$ is stable.
However, in the condensation phase, $T$ has only a single
fixed point, which is also stable.
See also Figure~\ref{fig:T} for an illustration.

\begin{figure}[htbp]
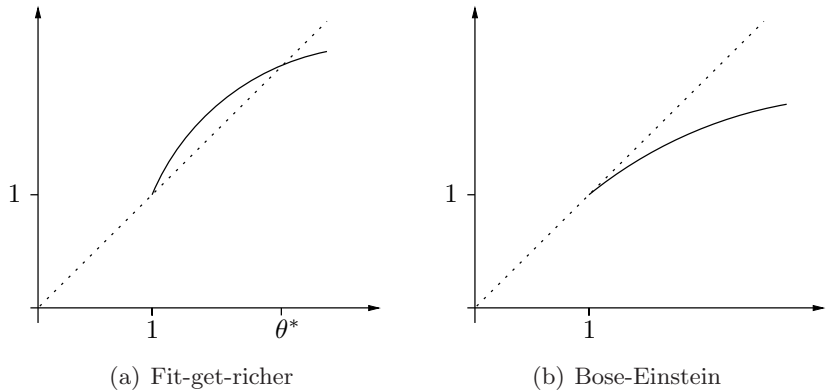

\centering 
\subfigure[Fit-get-richer]{ 	
		\input{FGR_plot.pspdftex}\hspace{2mm}
		}
\subfigure[Bose-Einstein]{\hspace{2mm} 	
		\input{BE_plot.pspdftex}
		}
\caption{The figure on the left shows the schematic graph 
of  $T$ in the case that $\nor^* > 1$, 
while the one on the right is for the case $\nor^* = 1$.}\label{fig:T}
\end{figure}

Section~\ref{se:deg_dist} is concerned with the proof of Theorem~\ref{thm:2}.
The proof is based on stochastic approximation techniques introduced in Section~\ref{se:prelim}.   In our setting these differential 
equations are non-linear because of the normalisation $\bar \calF_n$. 
However, since we can control the normalisation by Theorem~\ref{thm:main}, 
in the analysis of the joint fitness and degree distribution, we 
arrive at linear equations (or more precisely inequalities) for
the stochastic approximation. The latter then yield Theorem~\ref{thm:2}
via an approximation argument.

\section{Preliminaries}\label{se:prelim}

We first recall
the general idea of stochastic approximation, which goes back to~\cite{RobbinsMonro_1951}
and can be stated for example for a stochastic process 
${\bf X}_n$ taking values in $\R^d$. Then, ${\bf X}_n$ is known
as a stochastic approximation process, if it satisfies
a recursion of the type
\begin{equation}\label{eq:2401-1} {\bf X}_{n+1} - {\bf X}_n = \frac{1}{n+1} F({\bf X}_n) + {\bf R}_{n+1} - {\bf R}_n,\end{equation}
where $F$ is a suitable vector field and the increment of ${\bf R}$ corresponds
to an (often stochastic) error.
In our setting, we could for example restrict to the case when 
$\mu$ is supported on finitely many values $\{ f_1, \ldots, f_d\}\subset (0,1]$ and denote
by 
\[ X_n(k) = \frac{1}{n} \sum_{i=1}^n \calZ_n(i) \1_{\{ \calF_i = f_k \}} , \] 
the proportion of vertices that have fitness $f_k$ weighted by their impact.
Then, one can easily calculate the conditional expectation of $X_{n+1}(k)$
given the graph $\calG_n$ up to time $n$.
Indeed, as we will see in the proof of Proposition~\ref{prop:lower_emp}, under our assumptions we obtain that
\[ \E[ X_{n+1}(k) - X_n(k)\, |\,\calG_n ] 
 = \frac{1}{n+1} \Big( \mu(\{ f_k\}) + \frac{f_k}{\bar \calF_n}  X_n(k)- X_n(k)  \Big) 
\]
Therefore, we note that ${\bf X_n} = (X_n(k))_{k=1}^d$ 
satisfies
\[ X_{n+1}(k) - X_n(k)
 = \frac{1}{n+1} \Big( \mu(\{ f_k\}) + \frac{f_k}{\bar \calF_n}  X_n(k)- X_n(k) \Big)
+ R_{n+1}(k) - R_n(k) , 
\]	
so that ${\bf X}_n = (X_n(k))_{k=1}^d$ 
satisfies an equation of type~\ref{eq:2401-1}, provided
we take $R_{n+1}(k) - R_n(k) = X_{n+1}(k) - \E[ X_{n+1}(k) | \calG_n]$, 
which defines a martingale, for which we can employ the standard techniques to show
convergence. 

Provided that the random perturbations are asymptotically negligible, it is possible to analyse  the random dynamical system by the corresponding master equation
$$
\dot{\bf x}_t =F({\bf x}_t).
$$
There are many articles exploiting such connections and an overview is provided by~\cite{BenaimSurvey}. The connection to general urn models
is further explained in~\cite{PemantleSurvey}. 
In  random graphs, the resulting differential equation is closely related
to what is known as the master equation in heuristic derivations, see e.g.~\cite[Ch. 14]{newman_networks_2010}.

However, in our setting, this method is not directly applicable.
First of all, we would like to consider arbitrary fitness distributions
(i.e.\ not restricted to finitely many values)
and secondly the resulting equation is not linear, because of the appearance
of the normalization $\bar \calF_n$.
The latter problem is addressed by using a bootstrapping method (as described
in the introduction). However, this leads to an inequality
on the increment, rather than an equality as in~(\ref{eq:2401-1}).
Fortunately, the resulting vector field $F$ has a very simple structure and so
we can deduce the long-term behaviour of $\mathbf{X}_n$ 
by elementary means, the corresponding technical result 
is Lemma~\ref{le:stApp}.
By using inequalities, we also gain the flexibility to approximate arbitrary fitness 
distribution by discretization.

In order to keep our proofs self-contained, we will first state and prove 
an easy special case of the technique adapted to our setting.

\begin{lemma}\label{le:stApp}
 Let $(X_n)_{n \geq 0}$ be a non-negative stochastic process. We suppose that 
the following estimate holds
\begin{equation}\label{eq:2911-1} X_{n+1} - X_n  \leq \frac{1}{n+1} (A_n - B_n X_n) + R_{n+1}-R_{n}, \end{equation}
where
\begin{enumerate}
\item $(A_n)$ and $(B_n)$ are almost surely convergent stochastic processes with deterministic  limits $A,B>0$,
\item $(R_n)$ is an almost surely convergent stochastic process.
\end{enumerate}
Then one has that, almost surely,
\[ \limsup_{n \ra\infty} X_n \leq \frac{A}{B} . \]
Similarly, if instead under the same conditions (i) and (ii)
\[ X_{n+1} - X_n  \geq \frac{1}{n+1} (A_n - B_n X_n) + R_{n+1} -R_n, \]
then almost surely
\[ \liminf_{n \ra \infty} X_n \geq \frac{A}{B}. \]
\end{lemma}

\begin{proof} This is a slight adaptation of Lemma~2.6 in~\cite{PemantleSurvey}.
Fix $\delta \in (0,1)$. By our assumptions, almost surely, we can find $n_0$ such that for 
all $m,n \geq n_0$, 
\[ A_n \leq (1+\delta)A, \quad B_n \geq (1-\delta)B, \quad |R_m - R_n| \leq \delta .
\]
Then, by~(\ref{eq:2911-1}), we have that for any $m > n \geq n_0$,
\begin{equation}\label{eq:2911-2} \bal  X_m - X_n & \leq \sum_{j=n}^{m-1} \frac{1}{j+1} ( A_j - B_j X_j) 
 + |R_m-R_n| \\
& \leq \sum_{j = n}^{m-1} \underbrace{\frac{1}{j+1}((1+\delta)A - (1-\delta)B X_j))}_{=:Y_j} + \delta . 
\eal \end{equation}
Let $C = \frac{(1+\delta)A}{(1-\delta)B}$.
For each index $j\geq n_0$ with $X_j\geq C+\delta$, one has that $$Y_j\leq -B (1-\delta) \delta /(j+1).$$ Since the harmonic series diverges, by~(\ref{eq:2911-1}) there exists $m_0\geq n_0$ with $X_{m_0}\leq  C+\delta$. \smallskip

Next, we prove that for any $m\geq m_0$ one has $X_m\leq C+3\delta$ provided that $n_0$ is chosen sufficiently large (i.e. $\frac{1}{n_0 + 1}(1+\delta) A \leq \delta$). Suppose that $X_m> C+\delta$. We choose~$m_1$ as largest index smaller than $m$ with 
$X_{m_1}\leq C+\delta$. Clearly, $m_1\geq m_0$ and an application of estimate~(\ref{eq:2911-2}) gives
$$
X_m\leq  X_{m_1}+  Y_{m_1}+ \delta  \leq C+2\delta +\frac {1}{m+1} (1+\delta)A \leq C+3\delta = \frac {(1+\delta) A}{(1-\delta)B}+3\delta.
$$
Since $\delta\in(0,1)$ is arbitrary, we get that, almost surely,
$$
\limsup_{n\to\infty} X_n\leq \frac AB.
$$
The argument for the reverse inequality works analogously.
\end{proof}

As a first application of Lemma~\ref{le:stApp}, we can 
show that the total number of edges converges if properly
normalized.

\begin{lemma}\label{le:outdeg}
Almost surely, we have that
\[ \lim_{n \ra \infty} \frac{1}{n}\sum_{i =1}^n \cZ_n(i) = 1 + \la  .\] 
\end{lemma}

\begin{proof} 
Define $Y_n = \frac{1}{n} \sum_{i=1}^n \cZ_n(i)$. 
Then, we calculate the conditional expectation of $Y_{n+1}$ given $\calG_n$
using that $\calZ_{n+1}(n+1) = 1$ by definition as 
\[ \bal \E[ Y_{n+1} | \calG_n] & = 
\frac{1}{n+1} \big( \sum_{i=1}^n \E[ \calZ_{n+1}(i) | \calG_n ]  + 1\big) \\
& = Y_n  + \frac{1}{n+1} \big(1+ \sum_{i=1}^n\E[ \Delta Z_n(i) | \calG_n] - Y_n \big) \\
& = Y_n  + \frac{1}{n+1} (1+  \la
 - Y_n ) ,\\
\eal \]
where we used assumption (A1) on the conditional mean of $\Delta Z_n(i)$
and the definition of $\bar \calF_n$.
Thus, we can write
\begin{equation}\label{eq:2401-3} Y_{n+1} - Y_n = \frac{1}{n+1} ( 1 + \la - Y_n) + R_{n+1} - R_n  , \end{equation}
where we define $R_0 = 0$ and 
\[ \Delta R_n := R_{n+1} - R_n = Y_{n+1} - \E[ Y_{n+1} | \calG_n] . \]
Therefore, $R_n$ is a martingale and $R_n$ converges almost surely, 
if we can show
that $\E[ (\Delta R_n)^2 ]$ is summable.
Indeed, first using (A3) which states that impact evolutions of distinct
vertices are negatively correlated, we can deduce that
\[\bal  \E[ (\Delta R_n)^2 | \calG_n ] 
 & \leq \frac{1}{(n+1)^2} \bigl(\sum_{i =1}^n \E [ (\Delta \calZ_n(i) - \E[ \Delta \calZ_n(i) | \calG_n])^2 | \calG_n ] + 1\bigr) \\
& \leq \frac{1}{(n + 1)^2} \bigl( C^{\rm var} \sum_{i=1}^n \E[ \Delta \calZ_n(i) | \calG_n] + 1\bigr) \\
& \leq \frac{1}{(n+1)^2} ( C^{\rm var} \,\la + 1) , 
\eal \]
which is summable.

Hence, we can apply both parts of Lemma~\ref{le:stApp} together with the 
convergence of $(R_n)$ to obtain the almost surely convergence $\lim_{n \ra \infty} Y_n 
= 1 + \la$.
\end{proof}

Later on, we will need some a priori bounds on the normalisation sequence. 

\begin{lemma}\label{le:triv_bound} Almost surely, we have that 
 \begin{align*} \int x\, \mu(dx) \leq \liminf_{n\ra\infty} \tfrac{1}{n} \bar\calF_n  
\leq 
   \limsup_{n\ra\infty} \tfrac{1}{n} \bar \calF_n & \leq 1+ \la. 
 \end{align*}
\end{lemma}

\begin{proof} For the lower bound, notice that by definition $\cZ_n(i) \geq 1$,
and therefore
 \[ \liminf_{n \ra \infty} \frac{1}{n} \sum_{i=1}^n \calF_i \cZ_n(i) \geq 
\liminf_{n \ra \infty} \frac{1}{n} \sum_{i=1}^n \calF_i
  = \int x\, \mu(dx) . 
 \]
Conversely, one can use that the $\calF_i \leq 1$ and combine
with $\lim_{n \ra\infty} \frac{1}{n} \sum_{i=1}^n \cZ_n(i) = 1+ \la$, 
see also Lemma~\ref{le:outdeg}.
\end{proof}

\section{Proof of Theorem~\ref{thm:main}}\label{se:bootstrap}

The central bootstrap argument is carried out at the end of this section. It is based on Lemma~\ref{lemma:from_measure_to_normalisation}. Before we state and prove Lemma~\ref{lemma:from_measure_to_normalisation}, we prove a technical proposition which will be crucial in the proof of the lemma.

\begin{prop} 
\begin{itemize}
\item[(i)]\label{prop:lower_emp} Let $\nor \geq 1$. If  
\[ \limsup_{n\ra\infty} \bar {\calF}_n \leq \nor , \text{ almost surely}, \]
then for any $0\leq a< b\leq 1$, one has
\[ \liminf_{n\ra\infty} \frac{1}{n} \sum_{i = 1}^n \1_{\{\calF_i \in(a,b]\}}
\imp_{n}(i) \geq 
\int_{(a,b]} \frac \nor{\nor-f} \, \mu(df) 
, \]
almost surely.
\item[(ii)]\label{prop:up_emp}  
Let $\nor > 0$. If 
 \[ \liminf_{n \ra \infty} \bar \calF_n \geq \nor , \]
then for any $0\leq a < b < \nor\wedge 1$, 
\[ \limsup_{n\ra\infty} \frac{1}{n}\sum_{i=1}^n \1_{\{\calF_i \in (a,b]\}} \cZ_n(i)
 \leq \int_{(a,b]} \frac \nor{\nor-f} \, \mu(df). 
\]
\end{itemize}
\end{prop}

\begin{proof}
(i) First we prove that under the assumptions of $(i)$, one has for $0\leq f<f'\leq 1$, that
\begin{align}\label{eq:2401-2}
\liminf_{n\ra\infty} \frac{1}{n} \sum_{i = 1}^n \1_{\{\calF_i \in(f,f']\}}\,\cZ_n(i)\geq  \frac \nor{\nor-f} \, \mu((f,f']), \text{ almost surely.}
\end{align}
Let $0\leq f < f' \leq 1$ and
denote by 
\[ X_n = \Gamma_n((f,f']) = \frac{1}{n} \sum_{i \in \II_n} \cZ_n(i), \]
where we denote by $\II_n = \{ i \in \{1, \ldots,n\}\, : \, \calF_i \in (f,f']\}$. 

We will show~(\ref{eq:2401-2}) with the help of the stochastic approximation argument explained in Section~\ref{se:prelim}, see Lemma~\ref{le:stApp}. We need to provide a lower bound for the increment $X_{n+1} - X_n$. Using assumption (A1), we can calculate the conditional expectation of $X_{n+1}$:
\[ \bal \E[ X_{n+1} | \calG_n ]
 & = \frac{1}{n+1} \sum_{i \in \II_n} \E[\calZ_{n+1}(i) | \calG_n ] + \frac{1}{n+1} \p ( \calF_{n+1} \in (f,f']) \\
& = X_n + \frac{1}{n+1} \Big( \sum_{i \in \II_n} \E[ \Delta \calZ_n(i) | \calG_n] 
- X_n + \mu((f,f']) \Big) \\
& = X_n + \frac{1}{n+1} \Big( \sum_{i \in \II_n} \frac{\calF_i \cZ_n(i)}{n \bar \calF_n}
- X_n + \mu((f,f']) \Big) .\\
\eal \]
Hence, rearranging yields
\[ \E[ X_{n+1} |\calG_n] - X_n 
 \geq \frac{1}{n+1} \Big( \mu(f,f'] - \big( 1 - \frac{f}{\sup_{m \geq n} \bar \calF_m} \big) X_n \big) 
\]
Thus, we can write
\[ X_{n+1} - X_n \geq \frac{1}{n+1} \Big( \mu(f,f'] - \big( 1 - \frac{f}{\sup_{m \geq n} \bar \calF_m} \big) X_n \big)  + R_{n+1} - R_n, \]
where $R_n$ is a martingale defined via $R_0 =0$ and 
\[ \Delta R_n := R_{n+1} - R_n = X_{n+1} - \E[ X_{n+1} | \calG_n ] . \]
If we can show that $R_n$ converges almost surely, then Lemma~\ref{le:stApp}
together with the assumption that $\limsup_{n \ra \infty} \bar \calF_n \leq \theta$ shows 
that 
\[ \liminf_{ n \ra \infty} X_n \geq \frac{\theta}{\theta - f}\mu((f,f'])  , \]
which is the required bound~(\ref{eq:2401-2}). 

The martingale convergence follows if we show that $\E [ (\Delta R_n)^2 | \calG_n]$
is summable. Indeed,
\[ \Delta R_n = \frac{1}{n+1} \sum_{i \in \II_n} ( \cZ_{n+1}(i) - \E[\cZ_{n+1}(i) | \calG_n])
 + \frac{1}{n+1} ( \1_{\{ \calF_{n+1} \in (f,f']\}} - \mu((f,f'])  . 
\]
The second moment of the last expression is clearly bounded by $\frac{1}{(n+1)^2}
\mu((f,f'])$ which is summable, so we can concentrate on the first term. 
Now, we can use (A3), the negative correlation of $\Delta Z_n(i)$, and
then (A1) and (A2) to estimate the variance to deduce that
\[ \bal \frac{1}{(n+1)^2} & \E\Big[ \big.\Big( \sum_{i \in \II_n} ( \cZ_{n+1}(i) - \E[\cZ_{n+1}(i) | \calG_n]) \Big)^2
 \big| \calG_n \Big]\\
& \leq \frac{1}{(n+1)^2} \E\Big[ \big. \Big( \sum_{i \in \II_n} ( \Delta \cZ_n(i) - \E[\Delta \cZ_n(i) | \calG_n]) \Big)^2 
 \big|\calG_n \Big] \\
& \leq \frac{1}{(n+1)^2} \sum_{i \in \II_n} \Var ( \Delta \cZ_n(i)  | \calG_n ) \\
& \leq  \frac{1}{(n+1)^2}  C^{\rm var} \sum_{i \in \II_n} \frac{\calF_i \calZ_n}{n \bar \calF_n} \leq \frac{1}{(n+1)^2} C^{\rm var} \la ,
\eal \]
where we used the definition of $\bar \calF_n$ in the last step. The latter
is obviously summable, so that $R_n$ converges almost surely.

Note that the assertion (i) follows by a Riemann approximation. One partitions $(a,b]$ via $a=f_0<\dots<f_\ell=b$ with an arbitrary $\ell\iN$. Then it follows that
$$
\liminf_{n\to\infty} \frac 1n \sum_{i = 1}^n \1_{\{\calF_i \in(a,b]\}}
\imp_{n}(i) \geq \sum_{k=0}^{\ell-1}  \frac \nor{\nor-f_k} \, \mu((f_k,f_{k+1}]), \text{ almost surely},
$$
and the right hand side approximates the integral up to an arbitrary small constant.

(ii) It suffices to prove that for $0\leq f<f'<\nor\wedge 1$ one has 
\begin{align}\label{eq0502-1}
\limsup_{n\ra\infty} \frac{1}{n} \sum_{i = 1}^n \1_{\{\calF_i \in(f,f']\}}\,\cZ_n(i)\leq  \frac \nor{\nor-f'} \, \mu((f,f']), \text{ almost surely.}
\end{align}
This follows completely analogous to part (i) using Lemma~\ref{le:stApp}.
Then the statement (ii) follows as above by a Riemann approximation.
\end{proof}

The next lemma takes the lower bound on the fitness distribution obtained
in Proposition~\ref{prop:lower_emp} to produce a new upper bound on
the normalisation.
We set for $\nor\geq 1$
\begin{align}\label{eq0602-1}
T(\nor)= 1+ \frac 1\lambda \int  \frac {\nor-1}{\nor-f}\,f \,\mu(df)
\end{align}

\begin{lemma}\label{lemma:from_measure_to_normalisation} \begin{enumerate}\item[(i)] Let $\nor>1$. If  
$$\limsup_{n\to\infty} \bar \cF_n\leq \nor, \text{ almost surely,}$$
then 
$$\limsup_{n\to\infty} \bar \cF_n\leq T(\nor), \text{ almost surely,}$$
\item[(ii)]
Let $\nor> 0$ and suppose that
$$\liminf_{n\to\infty} \bar \cF_n\geq \nor, \text{ almost surely.}$$
One has, almost surely,
$$\liminf_{n\to\infty} \bar \cF_n
\geq  \left\{ \ba{ll} T(\nor) & \mbox{if } \nor \geq 1, \\
 \nor+\frac \nor\lambda (1-\mu[0,\nor)) & \mbox{if } \nor \in (0,1). \ea \right. 
$$
\end{enumerate}
\end{lemma} 

\begin{proof} (i)
Define a measure $\nu$ on $[0,1)$ via 
$$
\nu(df)= \frac {\nor}{\nor-f}\,\mu(df).
$$
Further, set $\nu'= \nu+ ((1+\lambda)-\nu[0,1))\delta_1$.
Since by Lemma~\ref{le:outdeg}, $\lim_{n\to\infty} \frac1n \sum_{i=1}^n \cZ_n(i)=1+\lambda$, almost surely, we get with Proposition~\ref{prop:lower_emp} that, for every $t\in(0,1)$,
$$
\limsup_{n\to\infty}  \frac{1}{n} \sum_{i=1}^n \1_{\{\calF_i \in [t,1) \}}\cZ_n(i) \leq  1+\lambda - \nu([0,t)) =\nu'([t,1]), \text{ almost surely.}
$$
This allows us to compute a new asymptotic upper bound for $(\bar\cF_n)$: let $m\iN$, observe that, almost surely,
\begin{align*}
\bar \cF_n&=\frac 1{\lambda n}  \sum_{i=1}^n \cF_i \,\cZ_n(i) \leq  \frac 1{\lambda n} \sum_{i=1}^n  \frac 1m \sum_{j=0}^{m-1} \1_{\{\cF_i\geq j/m\}} \,\cZ_n(i)\\
&=  \frac 1{\lambda m} \sum_{j=0}^{m-1} \frac 1n \sum_{i=1}^n  \1_{\{\cF_i\geq j/m\}} \,\cZ_n(i),
\end{align*}
so that
$$
\limsup_{n\to\infty} \bar\cF_n\leq  \frac 1{\lambda m}  \sum_{j=0}^{m-1} \nu'([j/m,1]), \ \text{ almost surely}.
$$
The latter expression tends with $m\to \infty$ to the integral $\frac 1\lambda \int x\,\nu'(dx)$ and we finally get that, almost surely,
$$
\limsup_{n\to\infty} \bar\cF_n \leq\frac 1\lambda \int f\,\nu'(df) = T(\nor).
$$

 (ii) Let $\nor'\in(0,\nor\wedge 1)$ and consider the (signed) measures $\nu=\nu(\nor')$ and $\nu'=\nu'(\nor')$ defined by
$$
\nu(df)=\frac {\nor}{\nor-f} \1_{[0,\nor']}(f) \, \mu(df)
$$
and
$$
\nu'= \nu +(1+\lambda-\nu[0,1])\delta_{\nor'}. 
$$
As above we conclude with Proposition~\ref{prop:lower_emp} that for $t<\nor'$, almost surely,
$$
\liminf_{n\to\infty} \sum_{i=1}^n \1_{\{\cF_i\in(t,1]\}} \cZ_n(i) \geq 1+\lambda -\nu((0,t]) = \nu'((t,1]).
$$
We proceed as above and note that for any $m \in \N$,
\[ \bal \bar \calF_n & = \frac{1}{\la n} \sum_{i=1}^n \calF_i \cZ_n(i) \geq \frac{1}{\la n} \sum_{i=1}^n
\frac{\nor'}{m}\sum_{i=1}^{m-1} \1_{\{ \calF_i \geq \frac{j}{m}\nor' \}} \cZ_n(i) \\
& = \frac{\nor'}{\la m}\sum_{j=1}^{m-1} \frac{1}{n} \sum_{i=1}^n \1_{\{ \calF_i \geq \frac{j}{m}\nor' \}} \cZ_n(i)
\eal \]
which yields that, almost surely, 
$$
\liminf_{n\to\infty} \bar\cF_n\geq \frac{\nor'}{\la m } \sum_{j=1}^m \nu'( ( \tfrac{j}{m}\nor' ,1]).
$$
Since $m\in\IN$ is arbitrary, we get that, almost surely,
$$
\liminf_{n\to\infty} \bar \calF_n\geq \frac 1\lambda \int f \,\nu'(df)=  \frac 1{\lambda} \Bigl(\nor'(1+\lambda) -\nor \int_{[0,\nor']} \frac{\nor'-f}{\nor-f} \mu(df)\Bigr).
$$
We distinguish two cases. If $\nor<1$, we use that  the latter integral is dominated by $\mu([0,\nor'])$ and let $\nor'\uparrow \nor$ to deduce that
$$
\liminf_{n\to\infty} \bar \calF_n\geq \nor+\frac \nor\lambda (1-\mu[0,\nor)), \text{ almost surely}.
$$
If $\nor\geq 1$, we let $\nor'\uparrow 1$ and get
$$
\liminf_{n\to\infty} \bar \calF_n\geq \frac 1\lambda \int f \,\nu'(df)=  \frac 1{\lambda} \Bigl(1+\lambda -\nor \int \frac{1-f}{\nor-f} \mu(df)\Bigr)=T(\nor).
$$
\end{proof}

Finally, we can prove Theorem~\ref{thm:main}, where we first show the normalisation
converges using a bootstrap argument based on Lemma~\ref{lemma:from_measure_to_normalisation}.
Finally we use the bound on the fitness distribution obtained in Proposition~\ref{prop:lower_emp} to show convergence of fitness distributions.

\begin{proof}[Proof of Theorem~\ref{thm:main}] \emph{(i) Fit-get-richer phase.} Suppose that $\nor^{**}$ is the smallest value in $[\nor^*,\infty)$ with
\begin{align}\label{eq0812-2}
\limsup_{n\to\infty} \bar \cF_n \leq \nor^{**}, \text{ almost surely}.
\end{align}
Such a value exists due to Lemma~\ref{le:triv_bound}.
We prove that $\nor^{**}=\nor^*$ by contradiction. Suppose that $\nor^{**}>\nor^*$. We apply Lemma~\ref{lemma:from_measure_to_normalisation} and get that
$$
\limsup_{n\to\infty} \bar\cF_n\leq T(\nor^{**}), \text{ almost surely}.
$$
Now note that $T$ is continuous on $[\nor^*,\nor^{**}]$ and differentiable on $(\nor^*,\nor^{**})$ with 
$$
T'(\nor)=\frac 1\lambda\int \frac{f(1-f)}{(\nor-f)^2}\,\mu(df)\leq \frac 1\lambda \int \frac f{(\nor-f)} \,\mu(df)<1.
$$
Further $\nor^*$ is a fixed point of $T$. Therefore, by the mean value theorem,
$$
T(\nor^{**})= T(\nor^*)+ T'(\nor) (\nor^{**}-\nor^*) <\nor^{**}
$$
for an appropriate $\nor\in(\nor^*,\nor^{**})$. This contradicts the minimality of $\nor^{**}$. 

We now turn to the convergence of the measures $\Gamma_n$. Note that the measure $\Gamma$ defined by
$$
\Gamma(df)= \frac {\nor^*}{\nor^*-f}\,\mu(df)
$$
has total mass $1+\lambda$. Since $\Gamma_n(0,1)=\frac 1n\sum_{i=1}^n \cZ_n(i)$ tends to $1+\lambda$, almost surely, one can apply the Portmanteau theorem to prove convergence of $(\Gamma_n)$. Let $\cD=\bigcup_{n\iN} 2^{-n}\IZ\cap[0,1]$ denote the dyadic numbers on $[0,1]$. We remark that the number of dyadic intervals $(a,b]$ with endpoints $a,b\in\cD$ is countable so that,  by Proposition~\ref{prop:lower_emp}, there exists an almost sure event $\Omega_0$, such that for all dyadic intervals $(a,b]$
$$
\liminf_{n\to\infty} \Gamma_n(a,b]\geq \Gamma(a,b] \  \text{ on } \ \Omega_0.
$$
Let now  $U\subset (0,1)$ be an arbitrary open set. We  approximate $U$ monotonically from within by a sequence of sets $(U_m)_{m\iN}$ with each $U_m$ being a union of finitely many pairwise disjoint dyadic intervals as above. Then, for any $m\iN$, one has
$$
\liminf_{n\to\infty} \Gamma_n(U) \geq \liminf_{n\to\infty} \Gamma_n(U_m) \geq\Gamma(U_m) \ \text{ on } \ \Omega_0
$$
and by monotone convergence, it follows that $\liminf_{n\to\infty} \Gamma_n(U) \geq\Gamma(U)$ on $\Om_0$. This proves convergence 
$$
\Gamma_n \Rightarrow \Gamma, \ \text{ almost surely.}
$$
Since $\bar \cF_n=\frac 1\lambda \int f \,\Gamma_n(df)$, we conclude that, almost surely,
$$
\lim_{n\to\infty} \bar\cF_n =\frac 1\lambda \int f \,\Gamma(df)= \nor^*.
$$

\emph{(ii) Bose-Einstein phase.} Let $\nor^*=1$. We start as in (i). Let $\nor^{**}$ denote the smallest value in $[1,\infty)$ with 
$$
\limsup_{n\to\infty} \bar\cF_n \leq \nor^{**}, \text{ almost surely.}
$$
As above  a proof by contradiction proves that $\nor^{**}=1$. Next, let $\nor^{**}$ denote the largest real in $(0,1]$ with
\begin{align}\label{eq0602-3}
\liminf_{n\to\infty} \bar\cF_n \geq \nor^{**}, \text{ almost surely.}
\end{align}
By Lemma~\ref{le:triv_bound}, such a $\nor^{**}$ exists and we assume that $\nor^{**}<1$. By Lemma~\ref{lemma:from_measure_to_normalisation}, the inequality (\ref{eq0602-3}) remains valid for
$$
\nor^{**}+\frac  {\nor^{**}}{\lambda}(1-\mu[0,\nor^{**}))>\nor^{**}
$$
contradicting the maximality of $\nor^{**}$. Hence,
$$
\lim_{n\to\infty} \bar\cF_n=1,\text{ almost surely}.
$$
By Proposition~\ref{prop:lower_emp}, one has, for $0\leq a<b<1$,
$$
\liminf_{n\to\infty} \Gamma_n(a,b] \geq \int_{(a,b]} \frac {1}{1-f}\,\mu(df) =\Gamma(a,b], \text{ almost surely},
$$
and, for $0\leq a <b=1$,
$$
\liminf_{n\to\infty} \Gamma_n(a,1] = 1+\lambda -\limsup_{n\to\infty}  \Gamma_n(0,a] =\Gamma(a,1], \text{ almost surely.}
$$
The rest of the proof is in line with the proof of (i).
\end{proof}

\section{Proof of Theorem~\ref{thm:2}}\label{se:deg_dist}

The proof is achieved via a stochastic approximation
technique as discussed in Section~\ref{se:prelim}.

\begin{proof}[Proof of Theorem \ref{thm:2}]
We prove the statement via induction over $k=1,2,\dots$. The proof of the initial statement ($k=1$) is similar to the proof  of the induction step and we will mainly focus on the latter task.\smallskip

Let $k\in\{2,3,\dots\}$ and suppose that the statement is true when replacing $k$ by a value in $1,\dots,k-1$. We fix $f,f'\in[0,1]$ with $\mu(\{f,f'\})=0$, $\mu((f,f']) > 0$ and consider the random variables 
$$
X_n :=\Gamma^\ssup{k}_n((f,f']) 
$$
for $n\in\N$. In the first step we derive a lower bound
for the increments of $(X_n)$ that is suitable for the application of Lemma~\ref{le:stApp}.

We restrict attention to vertices with fitness in $(f,f']$ and denote $\II_n:=\{i\in\{1,\dots ,n\}: \cF_i\in(f,f']\}$. Note that
\begin{align}\begin{split}\label{eq0401-1}
\E[X_{n+1}|\cG_n]&= \frac{1}{n+1}\sum_{i\in\II_n} \sum_{l=1}^k \ind_{\{\cZ_{n}(i)=l\}}\IP(\Delta \cZ_{n}(i)=k-l|\cG_n)\\
&=X_n+\frac{1}{n+1}\sum_{i\in\II_n}\Bigl( \sum_{l=1}^{k-1} \ind_{\{\cZ_{n}(i)=l\}}   \IP(\Delta \cZ_{n}(i)=k-l|\cG_n)\\
&  \ \ \ \ \ \ \ \ \ \ \qquad \qquad -  \ind_{\{\cZ_{n}(i)=k\}} \IP(\Delta \cZ_{n}(i)\not =0|\cG_n)\Bigr)- \frac{X_n}{n+1}.
\end{split}\end{align}
An application of the induction hypothesis gives that for fixed $l\in\{1,\dots,k-1\}$
$$
\lim_{n\to\infty}  \sum_{i\in\II_n}  \ind_{\{\cZ_{n}(i)=l\}}\IP(\Delta \cZ_{n}(i)=k-l|\cG_n) = \ind_{\{l=k-1\}} \, \frac {k-1}{\theta^*} \int_{(f,f']}  x\,\dd \Gamma^\ssup{k-1}(dx),
$$
almost surely. Indeed, for $l=k-1$
\begin{align*}
\Bigl|  \sum_{i\in\II_n}&  \ind_{\{\cZ_{n}(i)=k-1\}}\IP(\Delta \cZ_{n}(i)=1|\cG_n) -  \frac {k-1}{\bar\cF_n} \int_{(f,f']}  x\,\dd \Gamma^\ssup{k-1}(dx)\Bigr| \\
&\leq \sup_{i=1,\dots,n}  \ind_{\{\cZ_{n}(i)=k-1\}} n \Bigl| \IP(\Delta \cZ_{n}(i)=1|\cG_n)- \frac{(k-1)\cF_i}{n\,\bar\cF_n}\Bigr| \\
& \quad + \frac{k-1}{\bar \cF_n}\Bigl| \int_{(f,f']}  x\,\dd \Gamma_n^\ssup{k-1}(dx)-\int_{(f,f']}  x\,\dd \Gamma^\ssup{k-1}(dx)\Bigr| 
\end{align*}
and the former term tends to zero due to assumption (A4') and the latter term tends to zero by the induction hypothesis (and the fact hat $\Gamma^{(k-1)}$ puts no mass on $f$ and $f'$).
Analogously, one verifies the statement for the remaining $l$'s invoking assumption (A4). Further one has that
\[ \bal \sum_{i \in \II_n} & \1_{\{ \calZ_n(i) = k \}} \p ( \Delta \calZ_n(i) \neq 0 ) 
- \frac{f' k}{ \bar \calF_n} X_n 
\\
&  \leq \sum_{i \in \II_n} \1_{\{ \calZ_n(i) = k \}} \Big( \p ( \Delta \calZ_n(i) =1  | \calG_n )- \frac{\calF_i k}{n \bar \calF_n} \Big) + 
 \sup_{i=1,\ldots,n} n \p (\Delta \calZ_n(i) \geq 2 | \calG_n)
\eal
\]
where the two terms on the right hand side converge to $0$ by assumptions (A4) and (A4').

Consequently, there exist stochastic processes $(A_n)$ and $(B_n)$ such that
$$
\IE[X_{n+1}|\cG_n]-X_n\geq \frac 1{n+1} (A_n-B_n X_n)
$$
with $A_n\to \frac {k-1}{\theta^*} \int_{(f,f']}  x\,\dd \Gamma^\ssup{k-1}(x)$ and $B_n\to 1+\frac{k \, f'}{\theta^*}$, almost surely (where the former limit is positive
since $\mu((f,f']) > 0$). We now choose $(R_n)$ as the martingale with
$$
R_n=\sum_{k=1}^{n} (X_k-\E[X_k|\cG_{k-1}]) , 
$$
where we define $\calG_0$ as the empty graph and observe that 
\begin{equation}\label{eq:1801-1}
X_{n+1}-X_n \geq \frac1{n+1}(A_n-B_n X_n)+R_{n+1}-R_n.
\end{equation}

\emph{Convergence of the remainder term.} Next, we prove that $(R_n)$ converges almost surely.
An elementary calculation shows that the process $(R_n)$ is the difference of two martingales, namely
$$
M^\ssup{1}_{n+1}=M^{\ssup 1}_n+\frac1{n+1}\Bigl(\sum_{i\in\II_n} \ind_{\{\cZ_{n}(i)<k,\cZ_{n+1}(i) \geq k\}} -\E\big[\sum_{i\in\II_n} \ind_{\{\cZ_{n}(i)<k,\cZ_{n+1}(i) \geq k\}}|\cG_n\big]\Bigr),  
$$
and 
\begin{equation}\label{eq:1801-2}
M^\ssup{2}_{n+1}=M^{\ssup 2}_n+\frac1{n+1}\Bigl(\sum_{i\in\II_n} \ind_{\{\cZ_{n}(i)\leq k,\cZ_{n+1}(i) > k\}} -\E\big[\sum_{i\in\II_n} \ind_{\{\cZ_{n}(i) \leq k,\cZ_{n+1}(i)>k\}}|\cG_n\big]\Bigr)
 \end{equation}
both starting in $0$.  Since both martingales are the same up to a shift of parameter
$k$, we only have to show that either converges almost surely for fixed $k \in \N$.
We will show that $M^\ssup{2}$ converges by showing that its quadratic variation process converges almost surely. Indeed we will show that 
$\E [ ( \Delta M_n^\ssup{2} )^2 |\cG_n]$ is almost surely  summable, where $\Delta M_n^\ssup{2}
= M_{n+1}^\ssup{2} - M_n^\ssup{2}$.

First using assumption (A5), i.e.\ the conditional negative quadrant dependence of $\Delta Z_n(i)$, we find that
\[ \bal \E [ & (\Delta M_n^\ssup{2} ) ^2 | \calG_n ] \\
& \leq \frac{1}{(n+1)^2} \sum_{i \in \II_n} \E \Big[ \Big. \Big( 
\ind_{\{\cZ_{n}(i)\leq k,\cZ_{n+1}(i) > k\}} -\p \big( \cZ_{n}(i) \leq k,\cZ_{n+1}(i)>k\} | \calG_n \big) \Big)^2 \Big| \calG_n \Big] \\
& \leq \frac{1}{(n+1)^2} \sum_{i \in \II_n} \1_{\{ Z_n(i) \leq k \}} 
\p ( \Delta Z_n(i)  \geq 1 | \calG_n ) \\
& \leq \frac{1}{(n+1)^2} \sup_{i =1, \ldots,n} n 
\1_{\{ Z_n(i) \leq k \}} \Big| \p ( \Delta Z_n(i)  \geq 1 | \calG_n )
- \frac{\calF_i Z_n(i)}{ n \bar \calF_n} \Big| + \frac{\la}{(n+1)^2} ,
\eal 
\]
where we used the definition of $\bar \calF_n$ in the last step. By 
assumptions (A4) and (A4') the latter expression is  indeed almost surely summable.

\emph{Completing the induction step.}
Combining the convergence of the remainder term $R_n$ with the recursion
in~(\ref{eq:1801-1}), it follows with Lemma~\ref{le:stApp} that
\begin{align}\label{eq0601-1}
\liminf_{n\to\infty} \Gamma_n^\ssup{k}((f,f'])\geq \frac{(k-1) \int_{(f,f']} x \,d\Gamma^\ssup{k-1}(x)}{\theta^*+ k f'}
\end{align}
Recall that $f,f'\in[0,1]$ were chosen arbitrarily with $f<f'$ and $\mu(\{f,f'\})=0$,
where we can drop the assumption that $\mu((f,f']) > 0$, since the statement holds
trivially in that case. We now pick a countable subset $\IF\subset [0,1]$ that is dense such that for each of its entries $f$ one has $\mu(\{f\})=0$. The above theorem shows that there exists an almost sure set $\Om_0$ on which ~(\ref{eq0601-1}) holds for any pair $f,f'\in\IF$ with $f<f'$.
Suppose now that $U$ is an arbitrary open set. By approximating the set $U$ from below by unions of small  disjoints intervals $(f,f']$ with $f,f'\in\IF$ it is straight-forward to verify that
$$
\liminf_{n\to\infty} \Gamma_n^\ssup{k}(U)\geq (k-1) \int_{(f,f']} \frac{x}{\theta^*+ k x} \,d\Gamma^\ssup{k-1}(x)
$$
on $\Om_0$. The proof of the converse inequality, namely that almost surely, for any closed $A$, one has
$$
\limsup_{n\to\infty} \Gamma_n^\ssup{k}(A)\leq (k-1) \int_{(f,f']} \frac{x}{\theta^*+ k x} \,d\Gamma^\ssup{k-1}(x)
$$
is established in complete analogy. We thus obtain that $\Gamma_n^\ssup{k}$ converges almost surely, in the weak$^*$-topology to $\Gamma^\ssup{k}$ given by
$$
\Gamma^\ssup{k}(dx)= \frac {(k-1)x}{ k x + \theta^*} \,\Gamma^\ssup{k-1}(dx)=\prod_{l=2}^{k} \frac{(l-1)x}{lx + \theta^*}\, \Gamma^\ssup{1}(dx).
$$

\emph{Initialising the induction.}
To complete the argument, we still need to verify the statement for the initial choice $k=1$.
We again define $X_n = \Gamma_n^\ssup{1}((f,f'])$ for $f,f' \in [0,f]$
with $\mu(\{ f,f'\}) = 0$, $\mu(f,f'] > 0$ and we define $\II_n 
= \{ i \in \{1,\ldots,n\}: \calF_i \in (f,f']\}$. Then, 
it follows that since $\calZ_{n+1}(n+1) = 1$ by definition,
\[\bal \E[ X_{n+1} | \calG_n] & = \frac{1}{n+1} 
  \sum_{i \in \II_n} \p ( Z_{n+1} = 1 | \calG_n) + \frac{1}{n+1}\p ( \calF_{n+1} \in (f,f']) \\
& = X_n + \frac{1}{n+1} \Big[ \sum_{i \in \II_n} \1_{\{ Z_n(i) = 1\}}
\p ( \Delta Z_{n+1} \geq 1 | \calG_n)   - X_n  + \mu((f,f']) \Big] . 
\eal \]
Thus,
in complete
analogy with the induction step, one can show that 
\[ X_{n+1} - X_n \geq \frac{1}{n+1} ( A_n - B_n X_n) + R_{n+1} - R_n , \]
where $A_n \ra \mu((f,f'])$ and $B_n \ra 1 + \frac{f'}{\theta^*}$
and $R_{n+1} - R_n = X_{n+1} - \E [ X_{n+1} | \calG_n]$.
The remainder term $R_n$ can then be decomposed as $M^\ssup{1}_n - M^\ssup{2}_n 
+ I_n$ as above with $M^\ssup{1}_n = 0$, $M^\ssup{2}_n$ defined as in~(\ref{eq:1801-2})
and the additional term defined as
\[ I_n = \frac{1}{n} \big( \1_{ \{ \calF_n \in (f,f'] \}}  - \p ( \calF_n \in (f,f'] )\big) . \]
We have already seen that $M^\ssup{2}$ converges. Moreover, an elementary martingale argument (for i.i.d.\ random variables) shows
that $\sum_{n \in \N} I_n$ also converges almost surely.
Therefore, by Lemma~\ref{le:stApp} we can deduce that 
\[ \liminf_{n \ra\infty} X_n \geq \frac{\theta^*}{\theta^* + f'} \mu((f,f']) . \]
Repeating the same approximation arguments as before, we obtain that $\Gamma^\ssup{1}$
converges almost surely in the weak* topology to $\Gamma^\ssup{1}$ given by
\[ \Gamma^\ssup{1}(dx) = \frac{ \theta^* } {\theta^* + x} \mu( dx) , \]
which completes the proof by induction.
\end{proof}

%

\bibliographystyle{alpha}
\bibliography{PA_fitness}
\end{document}